\documentclass[a4paper,USenglish,cleveref]{lipics-v2021}

\usepackage{csquotes}
\usepackage{xspace}
\usepackage{bm}

\newcommand{\caseref}[1]{\hyperref[#1]{\textcolor{cgtGray}{\sffamily\bfseries{(\ref*{#1})}}\xspace}}

\title{Primal-Dual Cops and Robber}

\author{Minh Tuan Ha}{Karlsruhe Institute of Technology, Germany}{uwpwm@student.kit.edu}{}{}

\author{Paul Jungeblut}{Karlsruhe Institute of Technology, Germany}{paul.jungeblut@kit.edu}{https://orcid.org/0000-0001-8241-2102}{}

\author{Torsten Ueckerdt}{Karlsruhe Institute of Technology, Germany}{torsten.ueckerdt@kit.edu}{https://orcid.org/0000-0002-0645-9715}{}

\author{Pawe{\l} {\.{Z}}yli{\'{n}}ski}{University of Gda{\'{n}}sk, Poland}{pawel.zylinski@ug.edu.pl}{}{}

\authorrunning{M. T. Ha, P. Jungeblut, T. Ueckerdt and P. {\.{Z}}yli{\'{n}}ski}

\Copyright{Minh Tuan Ha, Paul Jungeblut, Torsten Ueckerdt, Pawe{\l} {\.{Z}}yli{\'{n}}ski}

\ccsdesc{Mathematics of computing $\rightarrow$ Discrete mathematics $\rightarrow$ Combinatorics $\rightarrow$ Combinatoric problems}

\keywords{Cops and robber, planar graph, dual graph}

\relatedversion{}
\relatedversiondetails[cite=Ha2023-PrimalDualEuroCG]{A preliminary version appeared at the EuroCG 2023.}{https://dccg.upc.edu/eurocg23/wp-content/uploads/2023/05/Booklet_EuroCG2023.pdf}
\relatedversiondetails[cite=Ha2024-PrimalDualCGT]{Full version published in Computing in Geometry and Topology.}{https://www.cgt-journal.org/index.php/cgt/article/view/49}

\nolinenumbers
\hideLIPIcs

\begin{document}

\maketitle

\begin{abstract}
    Cops and Robber is a family of two-player games played on graphs in which one player controls a~number of cops and the other player controls a robber.
    In alternating turns, each player moves (all) their figures.
    The cops try to capture the robber while the latter tries to flee indefinitely.
    In this paper we consider a variant of the game played on a planar graph where the robber moves between adjacent vertices while the cops move between adjacent faces.
    The cops capture the robber if they occupy all his incident faces.
    We prove that a constant number of cops suffices to capture the robber on~any planar graph of maximum degree~$\Delta$ if and only if~$\Delta \leq 4$.
\end{abstract}

\section{Introduction}
\emph{Cops and Robber} is probably the most classical combinatorial pursuit-evasion game on graphs.
The robber models an intruder in a network that the cops try to capture.
Two players play with complete information on a fixed finite connected graph~$G = (V,E)$.
The cop player controls a set of~$k$ cops, each occupying a vertex of~$G$ (possibly several cops on the same vertex), while the robber player controls a single robber that also occupies a vertex of~$G$.
The players take alternating turns, where the cop player in her\footnote{
    We use female pronouns for the cop player and male pronouns for the robber player.
} turn can decide for each cop individually whether to stay at its position or move the cop along an edge of~$G$ to an adjacent vertex.
Similarly, the robber player on his turn can leave the robber at its position or move it along an edge of~$G$.
The cop player starts by choosing starting positions for her~$k$ cops and wins the game as soon as at least one cop occupies the same vertex as the robber, i.e., when the robber is \emph{captured}.
The robber player, seeing the cops' positions, chooses the starting position for his robber and wins if he can avoid capture indefinitely.
The least integer~$k$ for which, assuming perfect play on either side,~$k$ cops can always capture the robber, is called the \emph{cop number} of~$G$, usually denoted by~$c(G)$.

In this paper, we introduce \emph{Primal-Dual Cops and Robber}, which is played on a plane graph~$G$, i.e., with a fixed plane embedding.
Here, the cops occupy the faces of~$G$ and can move between adjacent faces (i.e., along edges of the dual graph~$G^*$), while the robber still moves along edges between adjacent vertices of~$G$.
In this game, the robber is captured if \emph{every} face incident to his vertex is occupied by at least one cop.
Analogously, we call the least integer~$k$ for which~$k$ cops can always capture the robber in the Primal-Dual Cops and Robber game the \emph{primal-dual cop number} of~$G$ and denote it by~$c^*(G)$.

An obvious lower bound for~$c^*(G)$ is the maximum number of faces incident to any vertex in~$G$:
The robber can choose such a vertex as his start position and just stay there indefinitely (note that there is no \emph{zugzwang}, i.e., no obligation to move during one's turn).
In particular, if~$G$ has maximum degree~$\Delta(G)$ and there exists a vertex~$v$ with~$\deg(v) = \Delta(G)$, which is not a cut-vertex, then $c^*(G) \geq \Delta(G)$.
E.g., $c^*(K_{2,n}) = \Delta(K_{2,n}) = n$ for any~$n \geq 2$.

\subparagraph{Our contribution.}
We investigate whether the primal-dual cop number~$c^*(G)$ is bounded in terms of~$\Delta(G)$ for all plane graphs~$G$.
The answer is \enquote*{Yes} if $\Delta(G) \leq 4$ and \enquote*{No} otherwise.
In particular, our main result is the following theorem:

\begin{theorem}
    \label{thm:main}
    Each of the following holds:
    \begin{enumerate}
        \item\label{enum:degree_3}
        For every plane graph~$G$ with~$\Delta(G) = 3$ we have~$c^*(G) \leq 3$ and this is tight.

        \item\label{enum:degree_4}
        For every plane graph~$G$ with~$\Delta(G) = 4$ we have~$c^*(G) \leq 6$ and this is tight.
        \item\label{enum:degree_5}
        For infinitely many $n$-vertex plane graphs~$G$ with~$\Delta(G) = 5$ we have $c^*(G) = \Omega\bigl(\sqrt{\log n}\bigr)$.
    \end{enumerate}
\end{theorem}

\begin{remark}[Connectedness]
    The classical cop number~$c(G)$ (and many of its countless variants in the literature) are defined only for \emph{connected} graphs~$G$, as only the cops in the same component can capture the robber.
    However, in the Primal-Dual Cops and Robber game we do not require~$G$ to be connected, because the dual graph~$G^*$ of a planar graph~$G$ is always connected.
\end{remark}

\subparagraph{Related work.}
Let us just briefly mention that Cops and Robber was introduced by Nowakowski and Winkler~\cite{Nowakowski1983_VertexToVertex} and Quillot~\cite{Quilliot1978_Jeux} for one cop and Aigner and Fromme~\cite{Aigner1984_Planar3Cops} for any number of cops.
Since then numerous results and variants were presented, see e.g.~\cite{Bonato2022_Invitation,Bonato2011_TheGameOfCopsAndRobbers}.
Perhaps most similar to our new variant are the recent surrounding variant of Burgess et al.~\cite{Burgess2020_CopsThatSurroundARobber} with vertex-cops and the containment variant of Crytser et al.~\cite{Crytser2020_Containment,Pralat2015_Containment} with edge-cops.
In these variants the robber is captured if every adjacent vertex, respectively every incident edge, is occupied by a cop.
The smallest number of cops that always suffices for any planar graph~$G$ is three in the classical variant~\cite{Aigner1984_Planar3Cops}, seven in the surrounding variant~\cite{Bradshaw2019_SurroundingBoundedGenus}, $7\Delta(G)$ in the containment variant~\cite{Crytser2020_Containment} and~three when both, cops and robber, move on edges~\cite{Dudek2014_CopsAndRobberPlayingOnEdges}.
The surrounding variants with vertex- and edge-cops are not restricted to planar graphs and their respective cop numbers are always within a factor of~$\Delta(G)$ of each other, but they can be arbitrarily far apart from the classical cop number~\cite{Jungeblut2023_Surrounding}.
Finally, a wider perspective locates our problem as a
variant of the \emph{precinct} version of the Cops and Robber game, where each cop's movements are restricted to an assigned ``beat'' or subgraph, see~\cite{Clarke2002_PhD,Enright2023_MultiLayer,Fitzpatrick1997_PhD} for more details.

\section{Plane Graphs with \texorpdfstring{$\bm{\Delta = 3}$}{Maximum Degree 3}}

We start with an observation that simplifies the proofs of Statements~\ref{enum:degree_3} and~\ref{enum:degree_4} in \cref{thm:main}.

\begin{observation}
    \label{obs:robber_avoids_deg1_vertices}
    Let the robber be on a vertex~$u$ with a neighbor~$v$ of degree one.
    Then the~robber is never required to move to~$v$ to evade the cops.
\end{observation}

Clearly, this is true because the set of faces required to capture the robber at~$v$ is a subset of the faces required to capture him at~$u$.
Further, his only possible moves at~$v$ are either staying there or moving back to~$u$.
As there is no zugzwang, he could just stay at~$u$ all along.

\begin{proof}[Proof of Statement \ref{enum:degree_3} in \cref{thm:main}]
    Keeping in mind the lower bound of $c^*(G) \geq \Delta(G)$ for biconnected graphs, all we need is to provide a winning strategy for three cops, say~$c_1, c_2, c_3$, in a plane graph~$G$ with~$\Delta(G) \leq 3$.
    Each cop is assigned a \emph{target face}, initially the (up to) three faces incident to the robber.
    The goal of each cop~$c_i$ is to reach her \emph{target face}~$f_i$, thereby capturing the robber when all three cops arrived at their target faces.
    If the robber moves, each cop updates her target face.
    Our strategy guarantees that the total distance of all three cops to their target faces decreases over time, so it reaches zero after finitely many turns.

    Without loss of generality assume that~$G$ is a plane graph such that it contains only degree-$3$ vertices and degree-$1$ vertices (which can always be achieved by adding leaves to vertices not yet having the correct degree).
    First, the cops choose arbitrary faces to start on.
    Then the robber chooses his start vertex~$u$, which we assume to be of degree three by \cref{obs:robber_avoids_deg1_vertices} (it is trivial to capture him if all vertices are of degree one). 
    Let~$\angle^u_1, \angle^u_2, \angle^u_3$ be the three angles incident to~vertex $u$ and denote the face of~$G$ containing an angle~$\angle \in \{\angle^u_1, \angle^u_2, \angle^u_3\}$ by~$f(\angle)$.
    Now, for each cop~$c_i$, $i = 1,2,3$, we initially set~$f_i = f(\angle^u_i)$.
    
    Clearly, in every game the robber has to move at some point to avoid being captured.
    Assume now that the robber moves from vertex~$u$ to vertex~$v$ (both of degree three by \cref{obs:robber_avoids_deg1_vertices}).
    Without loss of generality assume that the angles around~$u$ and~$v$ are labeled as depicted in \cref{fig:upper_bound_max_degree3_angles}, with~$f_i = f(\angle^u_i)$ being the current target face of cop~$c_i$, $i=1,2,3$.
    
    \begin{figure}[tb]
        \centering
        \includegraphics{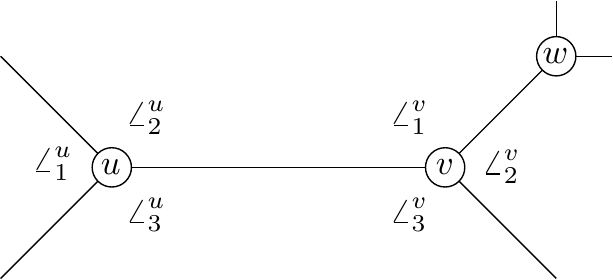}
        \caption{Labeling of the angles for a robber move from~$u$ to~$v$ (and possibly further to~$w$).}
        \label{fig:upper_bound_max_degree3_angles}
    \end{figure}
    
    Assume first that~$c_3$ (or symmetrically~$c_2$) has not reached her target face yet.
    In this case we assign the new target faces~$f_1 = f(\angle^v_1)$, $f_2 = f(\angle^v_2)$ and~$f_3 = f(\angle^v_3)$.
    Note that for~$i = 1,2$, faces~$f(\angle^u_i)$ and~$f(\angle^v_i)$ are adjacent, and so cop~$c_i$ can keep the distance to her target face unchanged (or even decrease it) during her next turn.
    Next, observe that~$f(\angle^u_3) = f(\angle^v_3)$, and so cop~$c_3$ can decrease her distance by one during her next turn.
    Thus, after the relevant moves, the total distance of the three cops to their target faces decreases by at least one.
    
    Assume now that~$c_2$ and~$c_3$ have already reached their target faces (but~$c_1$ has not, as the game would be over otherwise).
    If setting $f_1 = f(\angle^v_1)$ and $f_2=f(\angle^v_2)$ while keeping $f_3= f(\angle^u_3) = f(\angle^v_3)$ and then making the relevant moves toward target faces decreases the total distance of the three cops, we proceed with these assignments/moves.
    Otherwise, if the above setting is not satisfactory, we move~$c_1$ one step towards her target face~$f_1 = f(\angle^u_1)$ and~$c_2,c_3$ both to~$f(\angle^v_2)$.
    Now, its the robber's turn again.
    If he does not move, we assign target faces $f_i = f(\angle^v_i)$, $i=1,2,3$, and the total distance decreases after the cops' next turn.
    If he moves back to~$u$, we assign target faces $f_i = f(\angle^u_i)$, $i=1,2,3$, and the total distance also decreases after the cops' next turn.
    The last possibility for the robber is to move towards another neighbor~$w$ of~$v$, see \cref{fig:upper_bound_max_degree3_angles}.
    Then, we assign~$f_1 = f(\angle^v_1)$ and~$f_2,f_3$ to be the faces containing the other two angles at~$w$.
    In their next turn, $c_2$ and~$c_3$ can again reach their target faces, while~$c_1$ decreases the distance to her target face~$f(\angle^v_1)$ by one compared to the initial situation with the robber at vertex~$u$.
    Again, the total distance is decreased, which concludes the proof.
\end{proof}

\section{Plane Graphs with \texorpdfstring{$\bm{\Delta = 4}$}{Maximum Degree 4}}

To prove the upper bound in Statement~\ref{enum:degree_4} in \cref{thm:main}, we reduce our Primal-Dual Cops and Robber to the classical Cops and Robber with cops on vertices of the dual graph~$G^*$ of $G=(V,E)$ and then use a result from the literature. 

\subsubsection*{Upper Bound}
We claim that for every plane graph $G$ of maximum degree~$4$, we have $c^*(G) \leq 6$, i.e., that six face-cops can always catch the robber.
As a warming up, we first discuss a simple strategy that shows~$c^*(G) \leq 12$ for any plane graph~$G$ with~$\Delta(G) = 4$ and also gives rise to a better strategy.   
Without loss of generality assume that~$G$ is a plane graph such that it contains only degree-$4$ vertices and degree-$1$ vertices (which again can be achieved by adding leaves to vertices not yet having the required degree).
Following the notation in the proof of \cref{thm:main} (Statement~\ref{enum:degree_3}), consider a vertex-cop~$c$ located at a degree-$4$ vertex $u \in V$ with four incident angles~$\angle^u_i$, $i \in \{1,2,3,4\}$, and place four face-cops on the relevant four faces~$f(\angle^u_i)$.

\begin{remark}
    \label{rem:face_cops_simulate_vertex_cop}
    If the vertex-cop~$c$ moves to an adjacent vertex~$v$, the four face-cops around her can in one step also move to faces containing the angles incident to~$v$ (see \cref{fig:upper_bound_max_degree4_angles} for an illustration). 
 \end{remark}

Consequently, in any plane graph~$G$ with~$\Delta(G) = 4$, four face-cops can simulate a vertex-cop.
Aigner and Fromme prove that~$c(G) \leq 3$ for all plane~$G$~\cite{Aigner1984_Planar3Cops}, which immediately results in $c^*(G) \leq 4 \cdot c(G) = 12$ in that class of graphs. 

\begin{figure}[tb]
    \centering
    \includegraphics{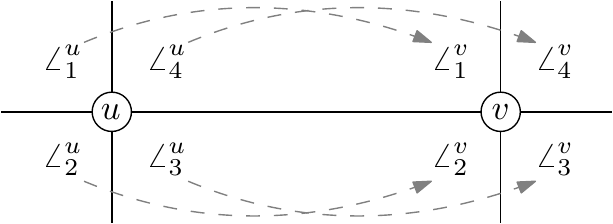}
    \caption{A vertex-cop and its four accompanying face-cops moving from~$u$ to~$v$.}
    \label{fig:upper_bound_max_degree4_angles}
\end{figure}

\medskip
But how about exploiting the fact that the dual graph~$G^*$ of~$G$ is planar?
Assume that the robber, initially located at vertex~$u$, with the relevant incident faces~$f(\angle^u_i)$, $i \in \{1,2,3,4\}$, is trying to avoid some six cops, say $c_1, \ldots, c_6$.
Let $U = \langle u_0=u, u_1, u_2, \ldots\rangle$ be the sequence of vertices in~$G$ visited by the robber.
Then, taking into account \cref{rem:face_cops_simulate_vertex_cop}, the sequence~$U$ defines the four sequences
\[
    F_i^* = \bigl\langle
        f^*(\angle_i^{u_0}),
        f^*(\angle_i^{u_1}),
        f^*(\angle_i^{u_2}),
        \ldots
    \bigr\rangle,
    \quad i \in \{1,2,3,4\}
\]
of vertices in the dual graph~$G^*$ of~$G$, corresponding to the four sequences 
\[
    F_i = \bigl\langle
        f(\angle_i^{u_0}),
        f(\angle_i^{u_1}),
        f(\angle_i^{u_2}),
        \ldots
    \bigr\rangle,
    \quad i \in \{1,2,3,4\}
\]
of faces incident to vertices~$u_j \in U$.  
Assume now that the cops' strategy is as follows:
First, cops~$c_1$, $c_2$ and~$c_3$, by playing together the Cops and Robber game in the dual graph~$G^*$ -- which is possible by \cref{rem:face_cops_simulate_vertex_cop} -- are trying to catch the imaginary robber~$r_1$ moving in~$G^*$ with respect to the sequence~$F^*_1$ (so cops~$c_4$, $c_5$ and~$c_6$ pause during this phase).

Clearly, by applying the strategy in~\cite{Aigner1984_Planar3Cops} for planar graphs, they will succeed in a finite number of steps.
Assume that the cop~$c_1$ is the one that catches the imaginary robber at vertex~$f^*_1(\angle_1^{u_{t_1}})$ at time moment~$t_1$.
Then, again taking into account \cref{rem:face_cops_simulate_vertex_cop}, the cop~$c_1$ can \emph{escort} (by being located at the same vertex in~$G^*$) the imaginary robber~$r_1$ forever.  
What do the remaining cops do next?
Cop~$c_4$ becomes activated and then cops~$c_2$, $c_3$ and~$c_4$, by applying the same strategy, catch the imaginary robber~$r_2$ in~$G^*$ which is (and has been) moving with respect to the sequence~$F^*_2$.
Again, after a finite number of steps, say at time moment~$t_2$, the face~$f(\angle_2^{u_{t_2}})$ becomes occupied by one of these cops, say~$c_2$.
Consequently, at time moment~$t_2$, while the (real) robber is located at vertex~$u_{t_2}$, the face(s)~$f(\angle_1^{u_{t_2}})$ and~$f(\angle_2^{u_{t_2}})$ are occupied by cops~$c_1$ and~$c_2$, respectively, which then can escort the imaginary robbers~$r_1$ and~$r_2$ forever (by \cref{rem:face_cops_simulate_vertex_cop}).

Observe now that, by activating cop~$c_5$ and eventually cop~$c_6$, again by applying the same strategy, all faces~$f_i(\angle_1^{u_{t_4}})$, $i \in \{1,2,3,4\}$, incident to the (real) robber will become occupied by four cops at some time moment~$t_4$, thus ending the game, which immediately results in~$c^*(G) \leq 6$ as required.

\subsubsection*{Lower Bound}

Consider a more restrictive variant of the classical Cops and Robber game in which both, cops and robber, play again on vertices and move along edges, and the robber is caught on~a~vertex~$v$ if and only if there is (at least) one cop on (at least) one neighbor~$u$ of~$v$.
So, while in the classical Cops and Robber game the robber must (in order to win) maintain distance at least~$1$ to all cops at all times, in the new variant we consider now, the robber must (in order to win) maintain distance at least~$2$ to all cops at all times.

It turns out, that in this variant played on a planar graph, also three cops are sometimes needed\footnote{
    Three cops are also sufficient, but we do not need this here.
}.
As in the classical game, a lower bound example can be derived from the dodecahedron graph~$D$.
In preparation for later arguments, let us prove this in the following stricter and, unfortunately, quite technical, form.

\begin{lemma}
    \label{lem:5_subdivided_dodecahedron}
    Let $D_5 = (V,E)$ be the dodecahedron graph~$D$ with each edge subdivided by five degree-$2$ vertices, as depicted in \cref{fig:subdivided_dodecahedron} with its unique plane embedding.
    Let two cops and one robber be playing on the vertices of~$D_5$ and moving along edges of~$D_5$.

    Then there is a strategy for the robber to maintain at all times distance at least~$2$ to both cops on~$D_5$.
    This still holds under the additional restriction that each time the robber moves onto a degree-$3$ vertex~$v$, he must announce (before seeing the cops' next move) to which neighbor of~$v$ he moves in his next turn.
\end{lemma}

\begin{figure}[tb]
    \centering
    \includegraphics{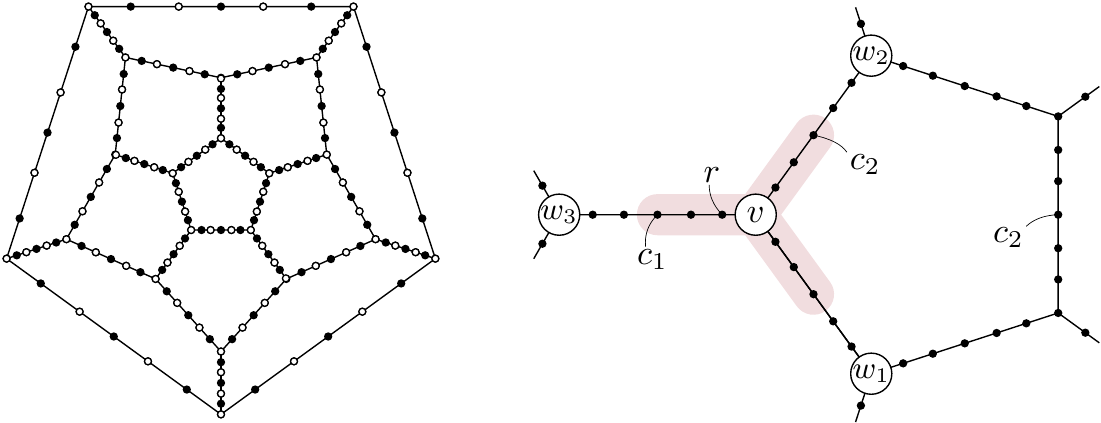}
    \caption{The $5$-times subdivided dodecahedron graph~$D_5$ (left) and an illustration of the strategy in \cref{lem:5_subdivided_dodecahedron} (right) with two exemplary positions for the cop $c_2$ at distance~$9$ to $w_1$.}
    \label{fig:subdivided_dodecahedron}
\end{figure}

\begin{proof}
    Let $V_3 = \{v \in V \mid \deg(v) = 3\}$ be the set of the~twenty degree-$3$ vertices in~$D_5$.
    Consider now a vertex $v \in V_3$.
    We say that a situation is \emph{safe around~$v$} if: 
    \begin{itemize}
        \item the robber has distance exactly~$1$ to~$v$ (i.e., occupies a neighbor of~$v$) and
        \item both cops have distance at least~$4$ to~$v$.
    \end{itemize}

    Our strategy for the robber goes as follows.
    After the initial placement of the two cops, the robber chooses his starting vertex to be a~neighbor, say~$s_v$, of some $v \in V_3$ to which both cops have distance at least~$4$, which makes him safe around~$v$\footnote{
        This can be done since each cop can have distance at most~$4$ to at most two vertices in~$V_3$.
    }.
    Then, the robber stays motionless at the vertex~$s_v$ as long as the situation is safe around~$v$.
    Otherwise, if at least one of the cops, say~$c_1$, enters a vertex at distance exactly~$3$ to~$v$, then the robber must react.
    He considers the three vertices $w_1,w_2,w_3 \in V_3$ at distance~$6$ to~$v$.
    Note that~$c_1$ has distance~$9$ to two of these, say~$w_1$ and~$w_2$.
    See the right of \cref{fig:subdivided_dodecahedron} for an illustration.
    Moreover, as each cycle in~$D_5$ is of the length at least~$30$, the other cop~$c_2$ has distance at least~$9$ to at least one of $w_1,w_2$, say to~$w_1$.

    Now, the robber moves from his current position at~$s_v$ to the closest vertex, say~$s_{w_1}$, in the neighborhood of~$w_1$.
    Clearly, going from~$s_v$ to~$s_{w_1}$ takes either four or six turns.
    In case the first step brings him to~$v$, the robber also announces his follow-up move, as required in the statement of the lemma.
    It follows from the choice of~$w_1$ that on the robber's way to~$s_{w_1}$, both cops have distance at least~$4$ to the robber.
    In addition, again by the choice of~$w_1$ and the fact that the cops made at most five steps and its their turn now, the situation is safe around~$w_1$ (for the robber at~$s_{w_1}$), which immediately completes the proof.
\end{proof}

Based upon~\cref{lem:5_subdivided_dodecahedron}, our purpose is to construct now a plane graph~$G$ of maximum degree $\Delta(G) = 4$ in which five face-cops are not enough to capture the robber, i.e., with $c^*(G) \geq 6$.
Before we start, we need one definition.
Let $H =( V,E)$ be a plane graph and consider a vertex $v \in V$.
Further, let the neighbors of~$v$ be ordered with respect to the plane embedding of~$H$, say $N(v) = (v_1, \ldots, v_k)$, where $k = \deg_H(v)$.
Then, the \emph{degree-based split of $v$} in~$G$ is the graph resulting from~$G$ by deleting~$v$, together with its incident edges, and adding~$k$ new vertices $u_1, \ldots, u_k$, together with the edges $\{v_1, u_1\}, \{u_1,v_2\}, \{v_2, u_2\}, \{u_2,v_3\}, \ldots, \{v_k, u_k\}, \{u_k,v_1\}$ forming a cycle of length $2k$.
See the top-right of \cref{fig:lower_bound_4_regular} for an illustration.

So let~$D$ be the dodecahedron graph and consider its plane subdivision~$D_5$ as in \cref{lem:5_subdivided_dodecahedron}.
Note that~$D_5$ is bipartite with one bipartition class~$A$ consisting of the set~$V_3$ of all degree-$3$ vertices of~$D_5$ together with two subdivision vertices per edge of the dodecahedron, and the other bipartition class~$B$ consisting of three subdivision vertices per edge of the dodecahedron, including all neighbors of vertices in~$V_3$.
Now, we construct our graph~$G$ from~$D_5$ as follows (see \cref{fig:lower_bound_4_regular} for an illustration). 
We first apply the degree-based split operation to all vertices $v \in A$.
Notice that since~$D_5$ is bipartite (in particular, $A$ is an independent set), the resulting graph~$G_0$ is unique, that is, the order in which we apply the split operations does not matter.
We then embed~$G_0$ in the plane, \enquote{inheriting} the embedding of~$D_5$, again see~\cref{fig:lower_bound_4_regular} for an illustration (here and there, we omit a formal construction since its description is tedious while we find it fairly enough illustrating it with a figure). 

It follows from the construction that for each face~$f$ of~$D$, the graph~$G_0$ contains a $30$-cycle~$C_f$ alternating between degree-$2$ vertices (corresponding to new vertices) and degree-$4$ vertices (corresponding to vertices in~$B$).
Now, for each such cycle~$C_f$, the next step is to connect (in a crossing-free way) its fifteen degree-$2$ vertices in order to create another~$30$-cycle, first turning these into fifteen degree-$4$ vertices and connecting them consecutively by adding new edges, and then subdividing each of these edges exactly once, thus creating fifteen new degree-$2$ vertices.
See the bottom-right of \cref{fig:lower_bound_4_regular}.
Starting with the new $30$-cycle, we repeat this step many times for each face~$f$ of~$D_5$ -- say~$100$ times, which is more than enough.
In the resulting graph~$G$, there is exactly one facial $30$-cycle for each face of~$D$.
We call these twelve faces of~$G$ the \emph{holes}. There are also twenty faces of length~$6$ in~$G$; one for each vertex of the dodecahedron $D$.
All remaining faces in~$G$ have length~$4$, see again \cref{fig:lower_bound_4_regular} for an illustration of the resulting graph~$G$.

\begin{figure}[tb]
    \centering
    \includegraphics[page=6]{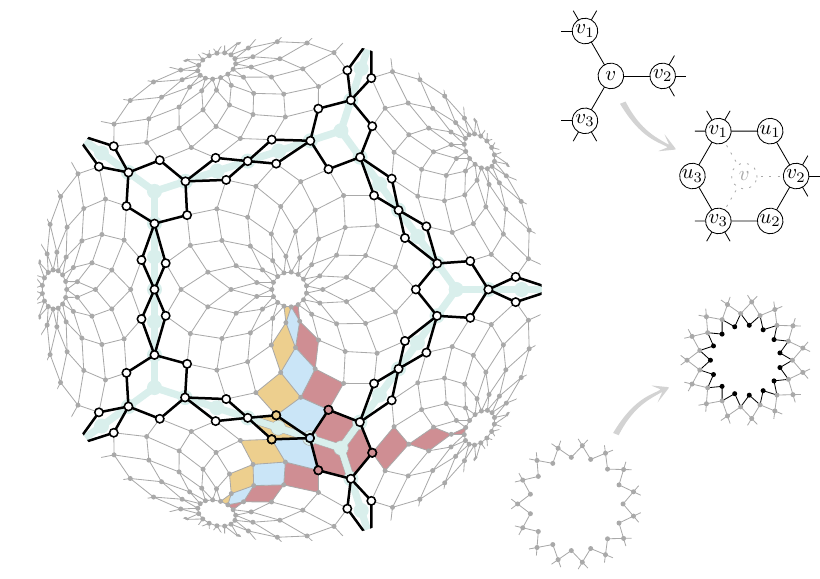}
    \caption{
        Left: A portion of the $4$-regular planar graph~$G$ based on~$D_5$.
        For the sake of readability only some $30$-cycles are shown in the faces of $D$.
        The subgraph~$G_0$ is highlighted in thick.
        Top-right: The degree-based split of $v$.
        Bottom-right: Adjoining a $30$-cycle to a $30$-cycle.}
    \label{fig:lower_bound_4_regular}
\end{figure}

To argue that~$c^*(G) > 5$, the robber pursues the following strategy against five face-cops.
He restricts himself to~$G_0$ only, focuses on just two of the cops~$c_1$ and~$c_2$, and ensures that neither~$c_1$ nor~$c_2$ ever gets on a face of~$G$ incident to the robber's current position.
As each vertex of~$G_0$ has four (pairwise distinct) incident faces, the three cops different from $c_1,c_2$ can never surround the robber and hence can be safely ignored.

In order to argue that the robber can maintain some distance to~$c_1$ and~$c_2$, we interpret any state of the game on~$G$ as a game on~$D_5$ in which the robber and the two cops $c_1,c_2$ occupy vertices of~$D_5$.

To be more precise, the robber moves on vertices of~$G_0$, each of which either corresponds directly to a vertex~$b \in B$ in~$D_5$ or to a vertex resulting from the degree-based split of some vertex $a \in A$ in~$D_5$.
On the other hand, the cops move on faces of~$G$:
\begin{itemize}
    \item The twenty $6$-faces correspond directly to vertices of degree~$3$ in~$D_5$.
    \item For each $4$-face, there is a well-defined corresponding vertex in~$D_5$ -- as indicated by the colors in \cref{fig:lower_bound_4_regular}-- such that for every pair of adjacent $4$- or $6$-faces in~$G$, the corresponding vertices in~$D_5$ form an edge.
    \item Only the holes have no corresponding vertex in~$D_5$.
\end{itemize}

We now observe that a vertex of~$G_0$ is incident to a $4$- or $6$-face in~$G$ only if the corresponding vertices in~$D_5$ have distance at most~$1$.
Thus, the strategy in \cref{lem:5_subdivided_dodecahedron} for the robber in~$D_5$ to maintain distance at least~$2$ to any two vertex-cops, provides a strategy for the robber in~$G$ to maintain distance at least~$2$ to any two face-cops, \emph{provided} the face-cops do not enter any hole\footnote{
    Note that it was crucial in \cref{lem:5_subdivided_dodecahedron} that the robber must announce his moves along a degree-$3$ vertex in~$D_5$, because correspondingly in~$G_0$ he moves to one of two possible vertices of the $6$-face.
}.
However, \emph{if} some face-cop, say~$c_1$, enters the hole in some face~$f$ of~$D_5$, it takes her at least~$100$ turns to reach again any face of~$G$ incident to a vertex in~$G_0$.
It is then easy for the robber to ignore~$c_1$ for~$100$ turns and to move, while maintaining distance at least~$2$ to~$c_2$, to a part of~$G_0$ \enquote{far away from~$f$}; for example to a part of~$G_0$ corresponding to a face of~$D_5$ that is non-adjacent to~$f$.

This proves that $c^*(G) > 5$ and therefore the tightness of Statement~\ref{enum:degree_4} in \cref{thm:main}.

\section{Plane Graphs with \texorpdfstring{$\bm{\Delta = 5}$}{Maximum Degree 5}}

In this section we prove Statement~\ref{enum:degree_5} in~\cref{thm:main}, i.e., that~$c^*(G) = \Omega\bigl(\sqrt{\log n}\bigr)$ for infinitely many $n$-vertex plane graphs~$G$ with~$\Delta(G) = 5$.
We utilize a result of Fomin et~al.~\cite{Fomin2010_FastRobber} about the cop number~$c_{p,q}(G)$ for a different variant of Cops and Robber for any graph~$G$ and positive integers~$p$ and~$q$.
Here (as in the classical variant) the cops and the robber are on the vertices of~$G$.
However, in each turn the cops may traverse up to~$p$ edges of~$G$, while the robber may traverse up to~$q$ edges of~$G$.
We refer to~$p$ and~$q$ as the \emph{velocities} of the cops and the robber, respectively.

\begin{theorem}[\cite{Fomin2010_FastRobber}]
    \label{thm:fast_robber}
    Let~$G_n$ be the $n \times n$ grid graph, $p$ be the velocity of the cops and~$q$ be the velocity of the robber.
    If~$p < q$, then~$c_{p,q}(G_n) = \Omega\bigl(\sqrt{\log n}\bigr)$.
\end{theorem}

The idea to prove Statement~\ref{enum:degree_5} in \cref{thm:main} is to construct a \enquote{grid-like} graph~$G_{n,s,r}$ for positive integers $n,s,r$ in which the robber in the primal-dual variant can move around faster than the cops.
Then he can simulate the evasion strategy of the robber in the variant of Fomin et al.~\cite{Fomin2010_FastRobber}.

We refer to \cref{fig:lower_bound_grid} for an illustration of the upcoming construction.
We start with the~$n \times n$ grid graph~$G_n$, $n \geq 3$, with a planar embedding such that the $4$-faces are the inner faces.
The outer face of $G_n$ has length $4(n-1)$.
We call the vertices of~$G_n$ the \emph{grid vertices}.
First, each edge of~$G_n$ is subdivided by~$2s$ new vertices, called \emph{subdivision vertices}, which results in the graph~$G_{n,s}$.
The inner faces of~$G_{n,s}$ have~$4\cdot 2s$ subdivision vertices on its boundary, while the outer face~$f_0$ has~$4(n-1)\cdot2s$. 
Next, inside each face~$f$ in~$G_{n,s}$ with $x$ subdivision vertices, we add~$r$ nested cycles (also nesting the boundary cycle of~$f$), called \emph{rings}, of length~$\frac32x$ each and call their vertices the \emph{ring vertices}.
Between any two consecutive rings we add a planar matching of~$\frac32x$ edges.
The \emph{closest ring} of an inner face~$f$ of~$G_{n,s}$ is the outermost ring in $f$; for the outer face~$f_0$ it is the innermost ring in~$f_0$.
At last, we add edges in a crossing-free way between each subdivision vertex~$v$ and (in total three) ring vertices on the closest ring in the two faces incident to~$v$ in~$G_{n,s}$ in such a way that:
\begin{itemize}
    \item each ring vertex on a closest ring has exactly one edge to a subdivision vertex,
    \item each subdivision vertex has two edges to ring vertices in one incident face of~$G_{n,s}$ and one edge to a ring vertex in the other incident face of~$G_{n,s}$, whilst
    \item the side with two edges to ring vertices always switches along each subdivision path.
\end{itemize}
In total, each subdivision vertex receives three edges, and each face of~$G_{n,s}$ with $x$ subdivision vertices receives $\frac32x$ edges which are connected to the~$\frac32x$ vertices of the closest ring.

\begin{figure}[tb]
    \centering
    \includegraphics[page=2]{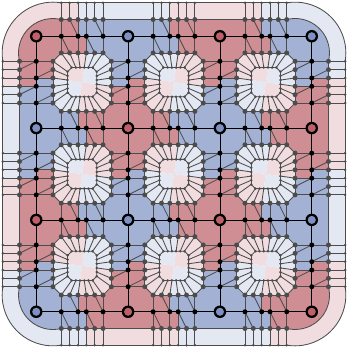}
    \caption{
        $G_{4,2,2}$: A $4 \times 4$ grid with each edge subdivided four times and two rings.
        Faces are colored according to their closest grid vertex.
        Deep and shallow faces are light and dark, respectively.
    }
    \label{fig:lower_bound_grid}
\end{figure}

Call the resulting graph~$G_{n,s,r}$ and note that~$\Delta(G_{n,s,r}) = 5$.
See again \cref{fig:lower_bound_grid} for an illustrating example.
We shall use a robber strategy in which he only focuses on grid vertices of~$G_{n,s,r}$ and moves between these through the paths of subdivision vertices, i.e., only plays on~$G_{n,s}$.
The purpose of the additional rings in~$G_{n,s,r}$ is to slow down the cops and force them to stay close to grid and subdivision vertices, too, thereby simulating the game of Fomin et al.~\cite{Fomin2010_FastRobber}~on~$G_n$.

Formally, we call a face of~$G_{n,s,r}$ \emph{shallow} if it is incident to some subdivision vertex, and \emph{deep} otherwise.
\Cref{lem:cop_moves_along_shallow_faces} below implies that, due to the number of rings, cops should not use deep faces.  

\begin{lemma}
    \label{lem:cop_moves_along_shallow_faces}
    Let~$a_1,a_2$ be two shallow faces of~$G_{n,s,r}$ inside the same face~$f$ of~$G_{n,s}$.
    If~$r > 3s(n-1)$, then any cop moving from~$a_1$ to~$a_2$ along a shortest path without leaving~$f$ uses only shallow faces.
\end{lemma}

\begin{proof}
    Let $x$ be the number of subdivision vertices on the boundary of~$f$.
    In particular,~$x \leq 4(n-1)\cdot 2s$.
    There are exactly~$\frac32x$ shallow faces inside~$f$; one for each edge of the closest ring in~$f$.
    Hence, the cop may move from~$a_1$ to~$a_2$ using only shallow faces in no more than~$\frac34x$ steps.
    On the other hand, the deep face~$b$ of length~$\frac32x$ is at distance~$r > 3s(n-1) \geq \frac38x$ from each of~$a_1,a_2$ and hence no shortest path between~$a_1$ and~$a_2$ uses~$b$.

    Let~$H$ be the subgraph of the plane dual of~$G_{n,s,r}$ induced by all inner faces inside~$f$, except~$b$.
    Then~$H \cong P_{r} \ \square\ C_{\frac32x}$ is a square grid on a cylinder of height~$r$ and circumference~$\frac32x$, with the shallow faces forming a boundary cycle~$C$.
    Since~$a_1,a_2$ are on~$C$ and each shortest path lies inside~$H$, such path is contained in~$C$, i.e., uses only shallow faces.
\end{proof}

Let $F$ be the set of all faces of~$G_{n,s,r}$.
For a face~$f \in F$, we denote by~$v_f$ the grid vertex closest to~$f$, breaking ties arbitrarily.

\begin{lemma}
    \label{lem:cop_moves_along_grid}
    Let~$a,b$ be two shallow faces whose closest grid vertices~$v_a,v_b$ have distance~$d$ in~$G_n$.
    If~$r > 3s(n-1)$, then in~$G_{n,s,r}$ the robber moving from~$v_a$ to~$v_b$ needs at most~$(2s+1)d$ steps, while any cop moving from~$a$ to~$b$ needs at least~$3s(d-2)$ steps.
\end{lemma}

\begin{proof}
    For the first part it is enough to observe that the robber may go along subdivision vertices, taking him exactly~$2s+1$ steps for every corresponding edge in~$G_n$.

    For the second part, i.e., the lower bound on the number of moves for a cop, let~$A$ and~$B$ be the inner faces of~$G_n$ containing the inner faces~$a$ and~$b$ of~$G_{n,s,r}$, respectively.
    We assume that~$d \geq 3$, as otherwise~$3s(d-2) \leq 0$ and there is nothing to show, and hence we have~$v_a \neq v_b$.
    More precisely, traveling from~$a$ to~$b$, the cop must traverse inner faces of~$G_{n,s,r}$ with at least~$d+1$ pairwise different closest grid vertices of~$G_n$.
    Cutting off the initial part with closest grid vertex~$v_a$ and final part with closest grid vertex~$v_b$, \cref{lem:cop_moves_along_shallow_faces} implies that the remaining shortest path for the cop uses only shallow faces.
    Thus, on her way, the cop visits shallow faces incident to at least~$d-1$ distinct grid vertices, i.e., $d-2$ transitions from a shallow face at a grid vertex to a shallow face at a neighboring (in $G_n$) grid vertex.
    As each such transition requires~$3s$ moves, the claim follows.
\end{proof}

\begin{proof}[Proof of Statement~\ref{enum:degree_5} in \cref{thm:main}]
    Fomin et al.~\cite{Fomin2010_FastRobber} describe an evasion strategy for a~robber with velocity~$q$ that requires~$\Omega\bigl(\sqrt{\log n}\bigr)$ vertex-cops with velocity~$p$ to capture him in~$G_n$, provided $q > p$; see \cref{thm:fast_robber}.
    In the following, we describe how a robber with velocity~$1$ in~$G_{n,s,r}$ (for sufficiently large~$n, s, r$) can simulate this strategy against face-cops with velocity~$1$.

    We choose~$p = 7$,~$q = 8$ and consider the game of Fomin et al.\ for these velocities.
    For their graph~$G_n$ in which the robber can win against $k = \Omega\bigl(\sqrt{\log n}\bigr)$ vertex-cops, we then consider $G_{n,s,r}$ with $s = 8$ and $r = 3s(n-1)+1$.
    Now we copy the evasion strategy~$\mathcal{S}$ for the robber as follows:
    Whenever it is the robber's turn and the face-cops occupy faces $f_1, f_2, \ldots, f_k$ in~$G_{n,s,r}$, consider the corresponding situation in~$G_n$ where the vertex-cops occupy vertices $v_{f_1}, v_{f_2}, \ldots, v_{f_k}$, respectively.
    Based on these positions,~$\mathcal{S}$ tells the robber to go to a~vertex~$v$ at distance~$d \leq q = 8$ from the current position of the robber in~$G_n$.
    By \cref{lem:cop_moves_along_grid}, the robber in~$G_{n,r,s}$ can go to~$v$ in at most $(2s+1)d \leq (2\cdot 8+1)\cdot 8 = 136$ turns.

    In the meantime, each face-cop also makes up to~$136$ moves in~$G_{n,r,s}$, traveling from some face~$a$ to some face~$b$, which is interpreted in~$G_n$ as the corresponding vertex-cop traveling from~$v_a$ to~$v_b$.
    For~$v_a$ and~$v_b$ to be at distance~$d' \geq 8$ in~$G_n$, the face-cop needs at least $3s(d'-2) \geq 3 \cdot 8 \cdot 6 = 144$ turns, by \cref{lem:cop_moves_along_shallow_faces}, which is strictly more than~$136$.
    Thus, after~$136$ turns, each vertex-cop made at most $p = 7$ steps in~$G_n$, as required for strategy~$\mathcal{S}$.

    Hence, the robber can evade~$k$ face-cops in~$G_{n,s,r}$, proving~$c^*(G_{n,s,r}) > k$.
    Since~$G_{n,s,r}$ for~$s \in O(1)$ and~$r \in O(n)$ has~$O(n^2)$ vertices, this completes the proof.
\end{proof}

\section{Conclusion and Discussion}
Let $c^*_\Delta$ denote the largest primal-dual cop number among all plane graphs with maximum degree~$\Delta$.
We have shown that $c^*_3 = 3$, $c^*_4 = 6$ (both bounds are tight), and $\textrm{sup}(c^*_5) = \infty$, while it is easy to see that $c^*_1 = 1$, $c^*_2 = 2$, and $\textrm{sup}(c^*_\Delta) = \infty$ for all $\Delta > 5$.
Let us remark that our proof for $\Delta = 5$ also holds for a variant of the game where the robber is already captured when a single cop reaches an incident face.
On the other hand, all our results for~$\Delta \leq 4$ can be also carried over to more restrictive variants in which the robber may not traverse an edge if one (or both) incident face(s) is/are occupied by a cop.

An interesting direction for future research would be to identify classes of plane graphs with unbounded maximum degree for which $c^*(G) \leq f(\Delta(G))$ for some function~$f$.
For example, what about plane Halin graphs or plane $3$-trees, also known as stacked triangulations?

Lastly, there are (at least) two possibilities to generalize Primal-Dual Cops and Robber to classes of potentially non-planar graphs.
\begin{itemize}
    \item Let~$G$ be a graph with a crossing-free embedding on any (orientable or non-orientable) surface~$\Sigma$.
    Our proof for the upper bound in the case~$\Delta(G) \leq 3$ still works.
    In particular, in that case, the primal-dual cop number is independent of the (non-orientable) genus of~$\Sigma$.
    Similarly, our result still holds for~$\Delta(G) \geq 5$ as the graph~$G_{n,s,r}$ constructed in the proof of \cref{thm:main} (Statement~\ref{enum:degree_5}) can be embedded on~$\Sigma$.
    Finally, turning our heads to~$\Delta(G) \leq 4$, the same strategy as that proposed in the proof of \cref{thm:main} (Statement~\ref{enum:degree_4}) can be used to give an upper bound of~$c(G^*) + \Delta(G) - 1$ on the primal-dual cop number in this case (where~$G^*$ is the dual graph corresponding to the given embedding on~$\Sigma$).
    In particular, when~$\Sigma$ is orientable with genus~$g$, it is known that~$c(G^*) \leq 1.268g$~\cite{Erde2021_BoundedGenus} (using the fact that~$G^*$ also has a crossing-free embedding on~$\Sigma$).
    
    \item A \emph{cycle double cover} of a graph~$G$ is a collection of cycles~$\mathcal{C}$ such that each edge of~$G$ is contained in exactly two cycles.
    For example, the set of a facial cycles of a biconnected plane graph forms a cycle double cover.
    It is famously conjectured that every biconnected graph has a cycle double cover.
    Given~$G$ and~$\mathcal{C}$, one can consider a variant of Cops and Robber with \emph{cycle-cops}, i.e., a cop always occupies a cycle of~$\mathcal{C}$ and can move in one turn from her current cycle $C \in \mathcal{C}$ to another cycle in $\mathcal{C}$ that shares an edge with $C$.
    The robber behaves as in the original game and is captured if all cycles incident to his vertex are occupied by cycle-cops.
    Our lower bound on the primal-dual cop number for graphs with~$\Delta \geq 5$ is still valid in this variant:
    the constructed graph~$G_{n,r,s}$ is planar and biconnected, and so its faces in the described planar embedding yield a cycle double cover.
    It is an interesting open question to consider the case~$\Delta \leq 4$ here.
\end{itemize}

\bibliographystyle{plainurl}
\bibliography{references}

\end{document}